\documentclass[12pt, leqno]{amsart}

\usepackage[OT2,T1]{fontenc}
\usepackage[utf8]{inputenc}
\usepackage{amstext}
\usepackage{amsthm}
\usepackage{amsopn}
\usepackage{amsfonts}
\usepackage{amsmath}
\usepackage{latexsym}
\usepackage[francais,english]{babel}
\usepackage{amscd}
\usepackage{amssymb}
\usepackage{amsmath}
\usepackage[all,cmtip]{xy}
\usepackage{leftidx}
\usepackage{graphicx}
\usepackage{etoolbox}

\usepackage{mathrsfs}
\patchcmd{\section}{\normalfont\scshape\centering}{\normalfont\bfseries}{}{}
\patchcmd{\subsection}{-.5em}{.5em}{}{}
\renewenvironment{proof}{{\noindent\bfseries Proof.}}{}

\newtheorem{theo}{{Theorem}}[section]
\newtheorem{coro}[theo]{{Corollary}}
\newtheorem{lemma}[theo]{{Lemma}}
\newtheorem{prop}[theo]{Proposition}

\theoremstyle{definition}
\newtheorem{remark}[theo]{\textbf{Remark}}

\newtheorem{example}[theo]{Example}
\numberwithin{equation}{section}
\newtheorem{notation}[theo]{Notation}
\newcommand{\ra}{\rightarrow}

\newcommand{\ol}{\overline}

\newcommand{\cE}{\mathcal{E}}

\newcommand{\cG}{\mathcal{G}}
\newcommand{\cH}{\mathcal{H}}

\newcommand{\cK}{\mathcal{K}}

\newcommand{\Br}{\mathrm{Br}}

\newcommand{\img}{\mathrm{Im}}

\newcommand{\Gal}{\mathrm{Gal}}

\newcommand{\inv}{\mathrm{inv}}

\newcommand{\gm}{\mathbb{G}}

\newcommand{\ff}{\mathfrak{f}}

\newcommand{\rat}{\mathbf{Q}}
\newcommand{\ent}{\mathbf{Z}}
\newcommand{\aff}{\mathbb{A}}

\DeclareFontEncoding{OT2}{}{} 
  \newcommand{\textcyr}[1]{%
    {\fontencoding{OT2}\fontfamily{wncyr}\fontseries{m}\fontshape{n}%
     \selectfont #1}}
\newcommand{\sha}{{\mbox{\textcyr{Sh}}}}
\newcommand{\ruB}{{\mbox{\textcyr{B}}}}

\begin{document}
\tolerance 400 \pretolerance 200 \selectlanguage{english}

\title{Norm tori of \'etale algebras and unramified Brauer groups}
\author{Eva  Bayer-Fluckiger and Ting-Yu Lee}
\date{\today}
\maketitle

\begin{abstract}
Let $k$ be a field, and let $L$ be an \'etale $k$-algebra of finite rank.
If $a \in k^{\times}$, let $X_a$ be the affine variety defined by
$N_{L/k}(x) = a$. Assuming that $L$ has at least one factor that is a cyclic field extension of $k$, we
give a combinatorial description of the unramified Brauer group of $X_a$.

\medskip


\medskip

\end{abstract}

\small{} \normalsize

\medskip

\selectlanguage{english}
\section{Introduction}

Let $k$ be a field, let $L$ be an \'etale $k$-algebra of finite rank, and let $N_{L/k} : L \to k$ be the norm map.
Let $a \in k^{\times}$, and let $X_a$ be the affine $k$-variety determined by
$${\rm N}_{L/k}(t) = a.$$

\smallskip
\noindent
Let  $X_a^c$ be a smooth compactification of $X_a$. The aim of this paper is to describe the  group
${\rm Br}(X_a^c)/{\rm Im}({\rm Br}(k))$  under the hypothesis that $L$ has {\it at least one cyclic factor}. We first
give a combinatorial description of a group associated to the \'etale algebra $L$ (see \S 3), and then
give an explicit isomorphism between this group and ${\rm Br}(X_a^c)/{\rm Im}({\rm Br}(k))$  (see \S 4, in
particular Theorem \ref{main theo}).

\medskip
Let us illustrate our results by a special case. Let $p$ be a prime number and let $n \geqslant 1$ be
an integer; assume that ${\rm char}(k) \not = p$ and
let $F$ be a Galois extension of $k$ with Galois group $\ent/p^n\ent\times\ent/p^n\ent$. Suppose that $L$ is
a product of $r$ linearly disjoint cyclic subfields of $F$ of degree $p^n$. Then we have (see Theorem
\ref{example of bicyclic}) :

\medskip
\noindent
{\bf Theorem.} $$\Br(X_a^c)/\img(\Br(k))\simeq(\ent/p^{n}\ent)^{r-2}.$$

\medskip We also give explicit generators of this group, as follows. With the above notation, let  $K$ be one of the cyclic subfields
of degree $p^n$ of $F$, and let $\chi$ be an injective morphism from ${\rm Gal}(K/k)$ to $\rat/\ent$. Let us
write $L = K \times K’$, with $K’ = \underset{i \in I} \prod K_i$, where $K_i$ is a cyclic subfield of $F$ of degree
$p^n$ of $F$ for all $i \in I$, and assume that $K$ and the fields $K_i$ are linearly disjoint in $F$. For all $i \in I$,
set $N_i = N_{K_i/k}(y_i)$, considered as elements of $k(X_a)^{\times}$. Assume that the cardinal of $I$ is $r-1$, so
that $L$ is a product of the $r$ linearly independent cyclic subfields $K$ and $K_i$ of $F$ of degree $p^n$.
Let $I'$ be a subset of cardinal $r-2$ of $I$.

\medskip
 Let $(N_i,\chi)$ denote the  class of the cyclic algebra over $k(X_a)$
associated to $\chi$ and the element $N_i \in k(X_a)^{\times}$.

\medskip
\noindent
{\bf Theorem.} The group $\Br(X_a^c)/\img(\Br(k))$ is generated by the elements $(N_i,\chi)$ for
$i \in I'$.

\medskip
This is also proved in Theorem
\ref{example of bicyclic}. Note that the above results are generalizations of \cite{BP}, Theorems 11.1 and 11.2.

\medskip

\medskip

The paper is organized as follows. Throughout the paper, $K$ is a finite cyclic extension of $k$, and $L = K \times K'$,
where $K'$ is an \'etale $k$-algebra of finite rank. Sections \ref {notation} and \ref {norm} are preliminary : in particular,
it is shown in \S \ref {norm} that we may assume $K$ be cyclic of prime power degree. Sections \ref{unramified} and
\ref{generators} contain the description of the unramified Brauer group. When $k$ is a global field, we obtain
additional results concerning the ``locally trivial" Brauer group (cf. \S \ref{global section}). Finally, in \S \ref{Hasse principles}
we apply Theorem \ref{main theo} to give an alternative proof of
\cite{BLP} Theorem 7.1 for
 $k$  a global field with $char(k)\neq p$; we show that the Brauer-Manin map of \cite{BLP} is the Brauer-Manin pairing,
 and hence deduce the Hasse principle from results of \cite{Sansuc} and \cite {DH}.

\section{Definitions and notation}\label{notation}

{\bf Generalities}

\medskip

Let $k$ be a field, let $k_s$ be a separable closure of $k$ and let $\cG_k = {\rm Gal}(k_s/k)$ be the
absolute Galois group of $k$.
 We fix once and for all this
separable closure $k_s$, and all separable extensions of $k$ that will appear in the paper will be contained in $k_s$.
We use standard notation in Galois cohomology; in particular, if $M$ is
a discrete $\cG_k$-module and $i$ is an integer  $\ge 0$, we set $H^i(k,M) = H^i(\cG_k,M)$. A {\it $\cG_k$-lattice} will
be a torsion free  $\bf Z$-module of finite rank on which $\cG_k$ acts continuously. For a $k$-torus $ T$, we denote by $\hat{ T} = {\rm Hom}(T,\gm_m)$ its character group; it is a $\cG_k$-lattice.

\medskip
 Let $G$ be a finite group. A {\it $G$-lattice} is by definition a ${\bf Z}$-torsion free ${\bf Z}[G]$-module of finite rank.
If $g \in G$, we denote by $\langle g \rangle$ the cyclic subgroup of $G$ generated by $g$. Let $M$ be
a $G$-lattice.  Set

$$\sha^2_{\rm cycl}(G,M) = {\rm Ker}[H^2(G,M) \to \prod_{g \in G} H^2(\langle g \rangle,M)].$$

\medskip
We recall a result of Colliot-Th\'el\`ene and Sansuc (cf. \cite{CTS 87}  Prop. 9.5)

\begin{theo}\label{Colliot Br cycl} Let $G$ be a finite group, let $T$ be a $k$-torus, and assume that the character group of $T$ is a $G$-lattice via a surjection $\cG_k \to G$. Let $T^c$ be a smooth compactification of $T$. We have ${\rm Br}(T^c)/{\rm Br}(k) \simeq \sha^2_{\rm cycl}(G,\hat T)$.

\end{theo}

\noindent
{\bf Proof.} See \cite{BP}, Theorem 1.3.

\medskip

{\bf Norm equations}

\medskip
Let $L$ be an \'etale $k$-algebra of finite rank; in other words, a product of a finite number of
separable extensions of $k$. Let $T_{L/{k}}  = R^{(1)}_{L/k}({\bf G}_m)$ be the $k$-torus
defined by $$1 \to T_{L/{k}} \to R_{L/k}({\bf G}_m) {\buildrel {\rm N}_{L/{k}} \over \longrightarrow}  {\bf G}_m \to 1.$$

\medskip Let $a \in k^{\times}$. Let $X_a$ be the affine $k$-variety associated to the {\it norm equation}

$${\rm N}_{L/k}(t) = a.$$

\medskip
The variety $X_a$ is a torsor under $T_{L/k}$;
let  $X_a^c$ be a smooth compactification of $X_a$. We have a natural map ${\rm Br}(k) \to {\rm Br}(X_a^c)$; if $a = 1$
then $X_1 = T_{L/k}$, and the map ${\rm Br}(k) \to {\rm Br}(T_{L/k}^c)$ is injective, and moreover we have an
injection  $${\rm Br}(X^c)/{\rm Im}({\rm Br}(k)) \to {\rm Br}(T^c)/{\rm Br}(k)$$ (see for instance \cite {BP}, \S 5).
Recall a result from \cite {BP}, Theorem 6.1 :

\begin{theo}\label{x = t}
 Assume that $L = K \times K’$, where $K/k$ is a cyclic extension and $K’$ an \'etale $k$-algebra.
Then the map ${\rm Br}(X^c)/{\rm Im}({\rm Br}(k)) \to {\rm Br}(T^c)/{\rm Br}(k)$ is an isomorphism.
\end{theo}

\medskip

{\bf Global fields}

\medskip
\medskip
If $k$ is a global field, we denote by
$V_k$ be the set of all places of $k$; if $v \in V_k$, we denote by $k_v$ the completion of $k$ at $v$.

\medskip
For any $k$-torus $T$, set $$\sha ^i(k,T) = {\rm Ker} (H^i(k,T) \ra \underset {v \in V_k} \prod H^i(k_v,T)).$$
If $M$ is a $\cG_k$-module,
set
$$\sha^i (k,M) = {\rm Ker} (H^i(k,M) \ra \underset {v \in V_k} \prod H^i(k_v,M)),$$and let $\sha_{\omega}^i(k,M)$ be the set of
$x \in H^i(k,M)$ that map to 0 in $H^i(k_v,M)$ for almost all $v \in V_k$.

\section {Norm equations and \'etale algebras}\label{norm}

In the sequel, we consider norm equations of \'etale algebras having at least one cyclic factor. The aim of
this section is to introduce some notation and prove some results that will be used throughout the paper.

\medskip
Let $K$ be a cyclic extension of $k$, and let $K’$ be an \'etale $k$-algebra of finite rank; set $L = K \times K’$. We first
show that it is enough to consider the case when $K/k$ is cyclic of prime power degree.

\medskip
{\bf Reduction to the prime power degree case}

\medskip
Let $\mathcal P$ be the set of prime numbers dividing $[K:k]$. For all $p \in \mathcal P$, let $K[p]$ be the
largest subfield of $K$ such that $[K[p]:k]$ is a power of $p$, and set $L[p] = K[p] \times K’$. Recall from
\cite {BLP} the following result

 \begin{prop}\label{reduction to prime power degree - sha} Assume that $k$ is a global
 field.  We have $$\sha^2(k,\hat T_{L/k}) \simeq \underset
{p \in \mathcal P} \oplus \sha^2(k, \hat T_{L[p]/k}).$$

\end{prop}

\noindent {\bf Proof.}
This follows from  \cite {BLP}, Lemma 3.1 and Proposition 5.16.

\medskip
Let $k’/k$ be
a Galois extension of minimal degree splitting $T_{L/k}$, and let $G = {\rm Gal}(k’/k)$.

\begin{prop}\label{reduction to prime power degree - cyclic sha}
 We have $\sha_{\rm cycl}^2(G,\hat T_{L/k})\simeq
 \underset
{p \in \mathcal P} \oplus \sha^2_{\rm cycl}(G,\hat T_{L[p]/k}).$

\end{prop}

\noindent
{\bf Proof.} Let us write $K’ = \underset{i \in I} \prod K_i$, where the $K_i$ are
finite field extensions of $k$.
Let $H$ be the subgroup
of $G$ such that $K = (k’)^H$, and for all $i \in I$, let $H_i$ be the subgroup of $G$ such that $K_i = (k’)^{H_i}$.
Set $M = \hat T_{L/k}$. We
have the exact sequence of $G$-modules

$$0 \to {\bf Z} \to  {\bf Z}[G/H]  \oplus \underset {i \in I} \oplus {\bf Z}[G/H_i] \to M \to 0.$$

For all $p \in \mathcal P$, let $H[p]$ be the subgroup of $G$ such that $K[p] = (k’)^{H[p]}$. Set $M[p] = \hat T_{L[p]/k}$.
We
have the exact sequence of $G$-modules

$$0 \to {\bf Z} \to  {\bf Z}[G/H[p]]  \oplus \underset {i \in I} \oplus {\bf Z}[G/H_i] \to M[p] \to 0.$$

\medskip
Let  $\ell’/\ell$ be an unramified extension of number fields with Galois group $G$ (cf. \cite{F}).
Set $L_0 = (\ell’)^{H}$, $L_0[p] = (\ell’)^{H[p]}$, and $L_i = (\ell’)^{H_i}$. Let $E = L_0 \times  \underset{i \in I} \prod L_i$
and $E[p] = L_0[p] \times  \underset{i \in I} \prod L_i$. We have $\hat T_{E/\ell} \simeq M$ and
$\hat T_{E[p]/\ell} \simeq M[p]$. By Proposition \ref{reduction to prime power degree - sha} we have
$\sha^2(\ell,M) \simeq \underset
{p \in \mathcal P} \oplus \sha^2(\ell, M[p]).$ Since $\ell’/\ell$ is unramified, we have $\sha_{\rm cycl}^2(G,M) \simeq \sha^2(\ell,M)$ and $\sha_{\rm cycl}^2(G,M[p]) \simeq \sha^2(\ell,M[p])$ (see \cite {BP}, Proposition 3.1), hence
$\sha_{\rm cycl}^2(G,M) \simeq \underset
{p \in \mathcal P} \oplus \sha^2_{\rm cycl}(G, M[p]).$

\begin{prop}\label{reduction to prime power degree - sha omega} Assume that $k$ is a global field We have

  $$\sha_{\omega}^2(k,\hat T_{L/k}) \simeq \underset
{p \in \mathcal P} \oplus \sha_{\omega}^2(k, \hat T_{L[p]/k}).$$

\end{prop}

\noindent
{\bf Proof.} This follows from Proposition \ref{reduction to prime power degree - cyclic sha} and \cite {BP}, Corollary 3.4.

\bigskip
{\bf The prime power degree case}

\medskip
Let $p$ be a prime number, and assume that {\it $K/k$ is cyclic of degree a power of $p$}. Let us write $K’ = \underset{i \in I} \prod K_i$, where the $K_i$ are
finite field extensions of $k$, and let $[K:k] = p^n$.

\begin{notation}\label {prime power notation} For all integers $1 \le m \le n$, let $K(m)$ be the unique subfield of $K$ of degree $p^m$ over $k$. The $K_i$-algebra $K(m) \otimes_k K_i$ is a product of cyclic extensions of $K_i$; let $p^{e_i(m)}$
be the degree of these extensions, and set $\mathcal E(m) = \{e_i(m) \ | \ i \in I \}$. For all $i \in I$, let
us chose one of the cyclic factors  $E_i/K_i$  of $K\otimes_k K_i$.
For all $m$ let $E_i(m)$ be the subfield of $ E_i$ which corresponds to a cyclic factor of $K(m)\otimes_k K_i$.

\medskip Let $\cK$ be a Galois extension of $k$ containing $K$ and all the fields $K_i$, and let
$G = {\rm Gal}(\cK/k)$. If $F$ is a subfield of $\cK$, we denote by $G_F$ the subgroup of $G$ such
that $F = \cK^{G_F}$.

\medskip
For all integers $0 \le m \le n$, let $\Gamma_i^m$ be the set of conjugacy classes of elements $g \in G$ such that
$\langle g\rangle \cap (G_{E_i(n-m)} \setminus G_{E_i(n-m+1)})\neq \emptyset$.

\end{notation}

\begin{notation}\label{global} Assume moreover that $k$ is a global field. Let $V_i^m$ be the set of places $v$ of $k$ such that there exists a place $w$ of $K_i$ above $v$ having the property that
$K \otimes_{k} (K_i)_w$ is a product of fields extensions of degree \emph{at least} $p^m$ of $(K_i)_w$.

\end{notation}

\begin{prop}\label{v et gamma}
Assume that $k$ is a global field.
Let $V_{rm}$ be the set of places of $k$ which are ramified in $\cK$.
For all integers $0 \le m \le n$, sending
a place $v \in V_i^m\setminus V_{rm}$ to the conjugacy class of its Frobenius element $\ff_v \in G$ gives rise to a surjection
from $V_i^m\setminus V_{rm}$ onto $\Gamma_i^m$.

\end{prop}

In order to prove the Proposition, we need the following lemma.

\begin{lemma}\label{inert}
 Let $F$ be a field, and let $E$ be a cyclic extension of $F$ of prime degree. Let $M$ be an extension of $E$,
 and assume that $M$ is a Galois extension of $F$.
Set $G_F = {\Gal}(M/F)$ and $G_E = {\Gal}(M/E)$. Let $v :  M^{\times} \to {\bf Z}$ be a discrete
valuation of $M$; assume that the restriction $v_F$ of $v$ to $F^{\times}$ is surjective, and that the residue field
of $v$ is perfect.

\medskip Let $D_{M/F}$
be the decomposition group of $v$. Then $v_F$ is inert in $E$ if and only if $D_{M/F} \cap (G_F \setminus G_E) \not =
\varnothing.$

\end{lemma}

\noindent
{\bf Proof.} Let $G_{E/F}$ be the Galois group of the extension $E/F$, and let $D_{E/F}$ be the decomposition group
of $v_F$; note that  $v_F$ is inert in $E$ if and only if $D_{E/F} = G_{E/F}$. Since $E/F$ is cyclic of prime degree,
this amounts to saying that $D_{E/F}$ is not trivial.

\medskip
We have the exact sequence $$1 \to G_E \to G_F \to G_{E/F} \to 1.$$
The image of $D_{M/F}$
by the homomorphism $G_F \to G_{E/F}$ is equal to $D_{E/F}$ (see for instance \cite {corps locaux}, Chap. I, Proposition 22). Hence $D_{E/F}$ is non trivial if and only if $D_{M/F} \cap (G_F \setminus G_E) \not =\varnothing.$

\begin {prop}\label{inert bis}
Let $F$, $E$ and $M$ be as in Lemma \ref{inert}. Assume moreover that
$k$ is a subfield of $F$, and that $M$ is a Galois extension of $k$. Let $G = \Gal(M/k)$.
Let $v_k : k^{\times} \to {\bf Z}$ be a discrete valuation such that the extensions of $v_k$ to $M$ are
unramified; let $\mathcal D$ be the set of corresponding decomposition groups. The following are equivalent

\medskip
{\rm (a)} There exists an extension of $v_k$ to $F$ that is inert in $E$.

\medskip
{\rm (b)} There exists $D \in \mathcal D$ such that $D \cap (G_F \setminus G_E) \not =
\varnothing.$

\end{prop}

\noindent
{\bf Proof.} Let us prove that (a) implies (b). Let $v_F$ be an extension of $v_k$ to $F$ that is inert in $E$,
let $v$ be an extension of $v_F$ to $M$, and let $D$ be the decomposition group of $v$. By Lemma
\ref{inert}, we have (b). Conversely, assume that (b) holds. Let $D \in \mathcal D$ be as in (b), and let $v$
be the corresponding valuation. Let $v_F$ be the restriction of $v$ to $F$. By Lemma \ref{inert}, we see
that $v_F$ is inert in $E$, hence (a) holds.

\medskip
\noindent
{\bf Proof of Proposition \ref{v et gamma}} If $v \in V_i^m\setminus V_{rm}$, then by definition there exists a place of $E_i(n-m)$ that is inert in the extension $E_i(n)/E_i(n-m)$,
and hence also in $E_i(n-m+1)/E_i(n-m)$; therefore by Proposition \ref{inert bis} the conjugacy class of its Frobenius element $\ff_v$ belongs to $\Gamma_i^m$.

Conversely, if the conjugacy class of $g\in G$ belongs to $\Gamma_i^m$,
then by Chebotarev's density theorem there exists an umramified place $v$ such that its Frobenius element $\ff_v$ is the conjugacy class of $g$.
Since the conjugacy class of $g$ belongs to $\Gamma_i^m$, there is a place $w$ of $E_i(n-m)$ above $v$ such that $w$ is inert in $E_i(n-m+1)/E_i(n-m)$.
Therefore $w$ is also inert in
$E_i(n)/E_i(n-m)$ as $E_i(n)/E_i(n-m)$ is cyclic of $p$-power degree and is unramified at $w$. This implies that $v \in V_i^m\setminus V_{rm}$.

\medskip

We get immediately the following corollary.
\begin{coro}\label{intersection of gamma}
For $i$, $j\in I$, the map defined in Proposition \ref{v et gamma} induces a surjection from $V_i^m\cap V_j^m\setminus V_{rm}$ to
$\Gamma^m_i\cap\Gamma^m_j$.
\end{coro}

\begin{remark}\label{preimage of gamma}
Keep the notation in Proposition \ref{v et gamma}.
For each conjugacy class in $\Gamma_i^m$, by Chebotarev's density theorem there are infinite many unramified places $v\in V_i^m$ mapped to it.
\end{remark}

\section {Norm equations - unramified Brauer group}\label{unramified}

We keep the notation of the previous section. In particular, $k$ is a field, $K$ is a cyclic extension of $k$, and $L = K \times K’$ where $K’$ is an \'etale $k$-algebra of finite rank.  Let $T_{L/{k}}  = R^{(1)}_{L/k}({\bf G}_m)$ be the $k$-torus
defined by $$1 \to T_{L/{k}} \to R_{L/k}({\bf G}_m) {\buildrel {\rm N}_{L/{k}} \over \longrightarrow}  {\bf G}_m \to 1.$$

\medskip Let $a \in k^{\times}$, and let $X_a$ be the affine $k$-variety associated to the norm equation
${\rm N}_{L/k}(t) = a.$
The variety $X_a$ is a torsor under $T_{L/k}$. Let $T_{L/k}^c$ be a smooth compactification of $T_{L/k}$, and
let  $X_a^c$ be a smooth compactification of $X_a$.

\medskip The aim of this section is to describe the group ${\rm Br}(X_a^c)/{\rm Im}({\rm Br}(k))$.
Using the results of \S \ref{norm}, we can assume that $K/k$ is of degree $p^n$, where $p$ is a prime number.

\medskip
We use the notation of \S \ref{norm} (see Notation \ref{prime power notation}). In addition, we need the following

\begin{notation}\label{group}

For all integers $n \ge 1$, we denote by $C(I,{\bf Z}/p^n{\bf Z})$ the set of maps $I \to {\bf Z}/p^n{\bf Z} $.

If $1 \le m \le n$, let
 $\pi_{n,m}$ be the projection $C(I,{\bf Z}/p^n {\bf Z}) \to C(I,{\bf Z}/p^m {\bf Z})$.

 For $x\in {\bf Z}/p^m\bf{Z}$ and $y\in{\ent}/p^r{\bf Z}$, we denote by $\delta(x,y)$  the maximum integer $d\le \min\{m,r\}$ such that
 $x=y$ (mod $p^d\bf{Z}$).

\end{notation}

We start with some special cases, in which the results are especially simple.

\bigskip

{\bf $K/k$ cyclic of degree $p$}

\medskip
Assume first that $[K:k]= p$, and that $K$ is not contained in any of the fields $K_i$. Then for all $i \in I$, $E_i$ is
a cyclic field extension of degree $p$ of $K_i$. Let
 $\Gamma_i = \Gamma_i^1$ be the set of conjugacy classes of elements $g \in G$ such that
$\langle g\rangle \cap (G_{K_i} - G_{E_i})\neq \emptyset $ (cf. Notation \ref{prime power notation}).

\medskip

Let $C(L)$ be the  group
$$\{c \in C(I,{\bf Z}/p {\bf Z}) \ | \ c(i) = c(j) \ {\rm if} \ \Gamma_i \cap \Gamma_j \not = \varnothing \},$$

\smallskip
\noindent  and $D$ be the subgroup
of constant maps $I \to {\bf Z}/p{\bf Z}$.

\medskip
As a consequence of Theorem \ref{general}, we’ll show the following

\begin{prop}\label{prime cyclic sha}   Assume that $K/k$ is cyclic of degree $p$, and that $K$ is not contained in
any of the fields $K_i$. Then we have

 $$\sha^2_{\rm cycl}(G,\hat T_{L/k}) \simeq C(L)/D.$$

 \end{prop}

By Theorem \ref{Colliot Br cycl}, this implies the following

\begin{coro} Assume that $K/k$ is cyclic of degree $p$, and that $K$ is not contained in
any of the fields $K_i$. Then we have

 $${\rm Br}(T^c)/{\rm Br}(k) \simeq C(L)/D.$$

\end{coro}

\bigskip
{\bf $K/k$ of degree $p^n$ and $K$ linearly disjoint of all the $K_i$}


\medskip
For all integers $m$ with $1 \le m \le n$ set

$$C^m = \{c \in C(I,{\bf Z}/p^{m} {\bf Z}) \ | c(i) = c(j) \ {\rm if} \
\Gamma_i^m \cap \Gamma_j^m \not = \varnothing \}.$$

\smallskip


\medskip
Set $$C(L) = \{ c \in C^n \ | \ \pi_{n,m}(c) \in C^m \ {\rm for} \ {\rm all} \ m \le n \},$$
\noindent
and denote by $D$  the subgroup of constant maps $I \to {\bf Z}/p^{n} {\bf Z}$.

\begin{prop}\label{prime power}  Assume that $K/k$ is cyclic of degree $p^n$, and that $K$ is linearly
disjoint of all the fields $K_i$. Then we have

$$\sha^2_{\rm cycl}(G,\hat T_{L/k}) \simeq C(L)/D.$$

 \end{prop}

 As in the case where $K/k$ is of degree $p$, this follows from Theorem \ref{general}, and has the immediate
 corollary

 \begin{coro} Assume that $K/k$ is cyclic of degree $p^n$, and that $K$ is linearly
disjoint of all the fields $K_i$. Then we have

$${\rm Br}(T^c)/{\rm Br}(k) \simeq C(L)/D.$$

 \end{coro}

 \bigskip
{\bf The general case}

\medskip
Recall that $K/k$ is cyclic of degree $p^n$, and that
we use Notation \ref {prime power notation}.
Recall that  $\mathcal E = \mathcal E(n)$.

\begin{notation}\label {general group}
For all $e \in \mathcal E$, set  $I_e = \{ i \in I \ | \ e_i(n) = e \}$. Denote by $\hat e$ the maximum element in $\cE$.
Note that the index $i$ belongs to  $I_e$ if and only if $K\cap K_i$ is an extension of degree $p^{n-e}$ of $k$.
As $K$ is a cyclic extension, this means that given  $0 \leq m \leq n$, the $e_i(m)$ are the same for all $i\in I_e$ and we denote it by $e(m)$.

For all integers $m$ with $1 \le m \le n$ set
$$C^m = \{c \in \underset {e \in \mathcal E} \oplus \  C(I_e,{\bf Z}/p^{e-e(n-m)} {\bf Z})\ | c(i) = c(j) \ {\rm if} \
\Gamma_i^m \cap \Gamma_j^m \not = \varnothing \}.$$

We still denote by $\pi_{ n,m}$  the map from $\underset {e \in \mathcal E} \oplus \  C(I_e,{\bf Z}/p^{e} {\bf Z})$ to $\underset {e \in \mathcal E} \oplus \  C(I_e,{\bf Z}/p^{e-e(n-m)} {\bf Z})$ induced by the natural projection.

\medskip
Set $$C(L) = \{ c \in C^n \ | \ \pi_{n,m}(c) \in C^m \ {\rm for} \ {\rm all} \ m \le n \},$$

\noindent
and denote by $D$ the image of constant maps $I \to {\bf Z}/p^{n} {\bf Z}$ in $C^n$ under the natural projection .
\end{notation}

\begin{remark}\label{e-e(n-m)}
If $E_i(n-m+1)\supsetneq E_i(n-m)$, then $K(n-m)\supseteq K\cap K_i$.
In this case $e_i(n)\geq m$ and $e_i(n)-e_i(n-m)=m$.
\end{remark}

The main results of this section are

\begin{theo}\label{general}  Assume that $K/k$ is cyclic of degree $p^n$.
Then we have

$$\sha^2_{\rm cycl}(G,\hat T_{L/k}) \simeq C(L)/D.$$

 \end{theo}

 By Theorem \ref{Colliot Br cycl}, this implies the following

 \begin{coro} Assume that $K/k$ is cyclic of degree $p^n$. then we have

$${\rm Br}(T^c)/{\rm Br}(k) \simeq C(L)/D.$$

\end{coro}

The proof of Theorem \ref{general} will be given below, using some arithmetic results of \cite{BLP}.
We start by recalling and developing some results concerning global fields.

\bigskip
{\bf Global fields}

\medskip
Assume that $k$ is a global field. Recall that $K/k$ is cyclic of degree $p^n$, and that we use notations
\ref{prime power notation} as well as \ref {general group}. In addition, for global fields, we also use notation
\ref{global}.

\bigskip
For all integers $m$ with $1 \le m \le n$ set

$$C_{\rm arith}^m = \{c \in \underset {e \in \mathcal E} \oplus \  C(I_e,{\bf Z}/p^{e-e(n-m)} {\bf Z})\ | c(i) = c(j) \ {\rm if} \
V_i^m \cap V_j^m \not = \varnothing \}$$ and
$$C_{\omega}^m = \{c \in \underset {e \in \mathcal E} \oplus \  C(I_e,{\bf Z}/p^{e-e(n-m)} {\bf Z})\ | c(i) = c(j) \ {\rm if} \
V_i^m \cap V_j^m {\rm \  is \ infinite}  \}$$


\medskip
Set $$C_{\rm arith}(L) = \{ c \in C_{\rm arith}^n \ | \ \pi_{n,m}(c) \in C_{\rm arith}^m \ {\rm for} \ {\rm all} \ m \le n \}$$
and

$$C_{\omega}(L) = \{ c \in C_{\omega}^n \ | \ \pi_{n,m}(c) \in C_{\omega}^m \ {\rm for} \ {\rm all} \ m \le n \}.$$

\begin{theo}\label{general global}
Assume that $K/k$ is cyclic of degree $p^n$. Then we have

\smallskip   {\rm (1)}  $\sha^2(k,\hat T_{L/k}) \simeq C_{\rm arith}(L)/D.$

\smallskip
 {\rm (2)}  $\sha_{\omega}^2(k,\hat T_{L/k}) \simeq C_{\omega}(L)/D.$

 \end{theo}

 \noindent
 {\bf Proof.}
 Recall some notation of \cite{BLP}.

 For $a=(a_i)\in\underset {e \in \mathcal E} {\oplus} \underset{i\in I_e}{\oplus}({\bf Z}/p^{e} {\bf Z})$ and $r\in\ent/p^{\hat{e}}\ent$, set $$I_r(a)=\{i\in I \ | \ a_i=r \ ({\rm mod}\  p^{e_i(n)}\ent)\}.$$

  Set
 \[G(K,K')=\{a=(a_i)\in\underset {e \in \mathcal E} \oplus \underset{i\in I_e}{\oplus}({\bf Z}/p^{e} {\bf Z}) \ | \ \underset {r \in {\bf Z}/p^{\hat{e}}{\bf Z}} \cap \ \underset { i \notin I_r(a)} \cup V_i^{\delta(r,a_i)+1} =\varnothing\} ;\]
  \[G_\omega(K,K')=\{a=(a_i)\in\underset {e \in \mathcal E} \oplus \underset{i\in I_e}{\oplus}({\bf Z}/p^{e} {\bf Z}) \ | \ \underset {r \in {\bf Z}/p^{\hat{e}}{\bf Z}} \cap \ \underset { i \notin I_r(a)} \cup V_i^{\delta(r,a_i)+1}\  \mbox{ is finite} \} .\]

 With the notation of \cite{BLP}, we have $\sha^2(k,\hat T_{L/k}) \simeq \sha(K,K’) = G(K,K')/D$,
 where $D$ is the subgroup generated by $(1,1,...,1)$ (see
\cite {BLP}, Theorem 5.3 and Lemma 3.1). Similarly, it is shown in \cite {Lee}, Theorem 2.5 that
$\sha_{\omega}^2(k,\hat T_{L/k}) \simeq G_{\omega}(K,K')/D$. Hence it suffices to show that $G (K,K')\simeq C_{\rm arith}(L)$
and that $G_{\omega}(K,K')\simeq C_{\omega}(L)$.

\medskip We show that $G(K,K')\simeq C_{\rm arith}(L)$; the proof of  $G_{\omega} (K,K')\simeq C_{\omega}(L)$ is the same.


Let $$f: \underset {e \in \mathcal E} {\oplus}\underset{i\in I_e}{\oplus}({\bf Z}/p^{e} {\bf Z})\to \underset {e \in \mathcal E} \oplus \  C(I_e,{\bf Z}/p^{e} {\bf Z})$$ be the map sending $(a_i)\in\underset{i\in I_e}{\oplus} ({\ent}/p^{e} \ent)$ to $c : I _e \to {\bf Z}/p^{e}{\bf Z}$ such that $c(i) = a_i$. We claim that the isomorphism $f$ gives rise to an isomorphism $$G (K,K')\to C_{\rm arith}(L).$$

For $c\in  \underset {e \in \mathcal E} \oplus \  C(I_e,{\bf Z}/p^{e} {\bf Z})$,
we denote $\pi_{n,m}(c)$ by $c_m$.

Let $a=(a_i)\in G(K,K')$ and $c=f(a)$.
We show that $c_m\in C^m$ for $1\leq m \leq n$.

Suppose that $V^m_i \cap V^m_j\neq \varnothing$.
By Remark \ref{e-e(n-m)}, we have $e_i-e_i(n-m)=e_j-e_j(n-m)=m$.
Let $v \in V^m_i\cap V^m_j$. As $a\in G(K,K')$, there is $r\in\ent/p^{\hat e}\ent$ such that
$v\notin \underset{l\notin I_r(a)}{\cup}V_l^{\delta(r,a_l)+1}$.
If $i\notin I_r(a)$, then $\delta(r,a_i)+1>m$ since $v\notin V_i^{\delta(r,a_i)+1} $. Hence $c(i)=r$ (mod $p^m\bf Z$).
If  $i\in I_r(a)$, then $c(i)=r$ (mod $p^{e_i(n)}\bf Z$).
In both cases we have $c_m(i)=r$ (mod $p^m\bf Z$).
The same argument works for $j$. Therefore $c_m(i)=c_m(j)$ and $c_m\in C^m$.

Let $c\in C^n$ such that $c_m\in C^m$ for  $1\le m\le n$.
Let $a=f^{-1}(c)$.
If $c\in D$, then clearly $a\in D$.
Suppose that $c\notin D$.
We claim that
$\underset {r \in {\bf Z}/p^{\hat{e}}{\bf Z}} \cap \ \underset { i \notin I_r(a)} \cup V_i^{\delta(r,a_i)+1} =\varnothing.$

Suppose not. Let $v\in \underset {r \in {\ent}/p^{\hat{e}}{\bf Z}} {\cap} \ \underset { i \notin I_r(a)} \cup V_i^{\delta(r,a_i)+1}$.
Choose $r_0\in \ent/p^{\hat{e}}\ent$. Since $c\notin D$, there is $r_1\in \ent/p^{\hat{e}}\ent$ and
$i\in I_{r_1}(a)\setminus I_{r_0}(a)$ such that
$v\in V_i^{\delta(r_0,a_i)+1}$.
For the same reason, there is $r_2\in \ent/p^{\hat{e}}\ent$ and
$j\in I_{r_2}(a)\setminus I_{r_1}(a)$ such that $v\in V_j^{\delta(r_1,a_j)+1}$.

By the choice of $i$ and $j$, we have $\delta(r_0,a_i)=\delta(r_0,r_1)$ and $\delta(r_1,a_j)=\delta(r_1,r_2)$.
Suppose that $\delta(r_0,r_1)\geq \delta(r_1,r_2)$.
Then $v\in V_i^{m}\cap V_j^{m}$, where $m=\delta(r_1,r_2)+1$.
Hence $c_m(i)=c_m(j)$ and $\delta(a_i,a_j)\geq m=\delta(r_1,r_2)+1$,
which contradicts that $\delta(r_1,r_2)\geq\delta(a_i,a_j)$.
Therefore $\delta(r_0,r_1)<\delta(r_1,r_2)$.

We can continue the above process to get an infinite sequence of $r_l\in\ent/p^{\hat e}\ent$ such that
$\delta(r_l,r_{l+1})<\delta(r_{l+1},r_{l+2})$. It is a contradiction as $\delta(r_l,r_{l+1})$ ranges from $0$ to $\hat{e}$. Hence $\underset {r \in {\bf Z}/p^{\hat{e}}{\bf Z}} \cap \ \underset { i \notin I_r(a)} \cup V_i^{\delta(r,a_i)+1} =\varnothing$ and $a\in G(K,K')$.  As a consequence $f$ induces an isomorphism $G(K,K')\to C_{arith}(L)$.

\begin{coro}\label{C and C omega}
Let $k$ be a global field. Then $\sha_\omega^2(k,\hat{T}_{L/k})\simeq C(L)/D$.
\end{coro}

\begin{proof}
By Corollary \ref{intersection of gamma} and
Remark \ref{preimage of gamma}, the two sets
$C^m$ and $C^m_\omega$ are the same.
Our claim then follows from Theorem \ref{general global}.
\end{proof}

\bigskip
\noindent
{\bf Proof of Theorem \ref{general}.} Recall that $\cK$ is a Galois extension of $k$  containing $K$ and all the fields $K_i$, and that
$G = {\rm Gal}(\cK/k)$; if $F$ is a subfield of $\cK$, we denote by $G_F$ the subgroup of $G$ such
that $F = \cK^{G_F}$.

Note that $k$ is not necessarily a global field here.
However there is always an unramified extension $\ell'/\ell$ with Galois group $\Gal(\ell'/\ell)\simeq G$ (\cite{F}).
Hence we can regard $\hat{T}_{L/K}$ as a $\Gal(\ell'/\ell)$-module.

To be precise, set $F = (\ell')^{G_K}$, $L_i = (\ell')^{G_{K_i}}$ and
$E = F \times \underset{i \in I} \prod L_i$. By construction, the $G$-lattices $\hat T_{E/\ell}$ and $\hat T_{L/k}$
are isomorphic.

Since the extension $\ell'/\ell$ is unramified, we have
$\sha^2(\ell,\hat T_{E/\ell})\simeq \sha_\omega^2(\ell,\hat T_{E/\ell}) \simeq \sha^2_{\rm cycl}(G,\hat T_{E/\ell})$.
By Corollary \ref{C and C omega}
the group $\sha^2_{\rm cycl}(G,\hat T_{E/\ell})$ is isomorphic to $ C(E)/D$.
However  $C(L)$ only depends on the group $G$.
Hence $C(L)\simeq C(E)$.
Therefore $\sha^2_{\rm cycl}(G,\hat T_{L/k})\simeq\sha^2_{\rm cycl}(G,\hat T_{E/\ell})\simeq C(L)/D$.

\section{Unramified Brauer groups and generators}\label{generators}

We keep the notation of the previous sections. Recall that $p$ is a prime number, $K/k$ a cyclic field extension
of degree $p^n$, and $L = K \times K’$, where $K’$ is an \'etale $k$-algebra of finite rank. In the previous section,
we introduced a group $C(L)$ and proved that ${\rm Br}(X_a^c)/{\rm Im}({\rm Br}(k)) \simeq C(L)/D$.

\medskip
The aim of this section is to give more precise information about the isomorphism $C(L)/D \to {\rm Br}(X_a^c)/{\rm Im}({\rm Br}(k))$.

\medskip Let ${\rm Br_{ur}}(k(X_a))$ be the subgroup of ${\rm Br}(k(X_a))$  consisting of all elements which are unramified at all discrete valuations of $k(X_a)$ with residue fields containing $k$ and with fields of fraction $k(X_a)$; recall that ${\rm Br_{ur}}(k(X_a))$ is isomorphic to ${\rm Br}(X_a^c)$ (see Cesnavius \cite {C}, Theorem 1.2).

\medskip As in the previous sections, let us write
$K’ = \underset{i \in I} \prod K_i$, where the $K_i$ are
finite separable field extensions of $k$.

\begin{notation}\label{galois groups and R}
We denote by $\cG_k$ the absolute Galois group of $k$, $\cG_{k(X_a)}$ the absolute Galois group of
$k(X_a)$.
Let $R$ be a discrete valuation ring of $k(X_a)$ with residue field $\kappa_R$ containing $k$ and with field of fractions $k(X_a)$.
We denote by $\cG_R$ the absolute Galois group of $\kappa_R$.
\end{notation}

\begin{notation}\label{definition of norm and Brauer group}
For all $i \in I$, let $\{\beta_{ij}\}$ be a basis of $K_i$ over k.
Let $y_i=\underset{j}{\sum} \beta_{ij}x_{ij}$, where $x_{ij}$ are variables.
Set $N_i = N_{K_i\otimes k(X_a)/k(X_a)}(y_i)$ considered as an element of $k(X_a)^{\times}$.
We define $N=N_{K\otimes k(X_a)/k(X_a)}(y)$ in a similar way.
Fix an isomorphism $\chi:{\rm Gal}(K/k)\to \ent/p^n\ent$.
Then $\chi$ gives rise to  a morphism $\tilde{\chi}: \cG_{k(X_a)}\to \ent/p^n\ent$ and
a morphism $\chi_R:\cG_{R}\to\ent/p^n\ent$.
Let $(N_i,\tilde\chi)$ denote the  class of the cyclic algebra over $k(X_a)$
associated to $\tilde\chi$ and the element $N_i \in k(X_a)^{\times}$. (\cite{GS} Prop. 4.7.3)
\end{notation}

\medskip
The main result of this section is

\begin{theo} \label{main theo}
Assume $char (k)\neq p$. Then the map $$ u : C(L) \to {\rm Br}(k(X_a))$$ given by $$u(c) = \underset {i \in I} \sum c(i)(N_i,\tilde\chi) $$ induces
an isomorphism

$$C(L) /D\to {\rm Br}(X_a^c)/{\rm Im}({\rm Br}(k)).$$
\end{theo}

\medskip

\begin{remark}\label{chi and gamma}
Note that   $(N_i,\tilde{\chi})\in \Br(k(X_a))$ has order at most $p^{e_i(n)}$, so  $c(i)(N_i,\tilde\chi) $ is well-defined for $c(i)\in\ent/p^{e_i(n)}\ent$  (ref.  \cite{BLP} Lemma 6.1).
\end{remark}

We start with following lemmas.
\begin{lemma}\label{diagonal vanish}
The group $u(D)$ is contained in the image of $\Br(k)$ in $\Br(k(X^c))$.
\end{lemma}
\begin{proof}
Since $N\cdot\underset{i\in I}\prod N_i=c$,
we have  $\underset{i\in I}{\sum}(N_i,\tilde{\chi})=(c/N,\tilde{\chi})=(c,\tilde{\chi})$,
which is the image of $(c,\chi)$ in $\Br(k(X_a))$.
Hence $u(D)\subseteq \img(\Br(k))$.

\end{proof}

\medskip

The following lemma can be found  in \cite{Lee} \S3.
Here we use the notation $C(L)$ to simplify the proof.

\begin{lemma}\label{definition of c'}

\noindent
\begin{enumerate}
\item  Let $c\in C(L)\setminus D$.
Pick $i_0\in I_{\hat{e}}$.
Let $\hat{c}\in D$ be the image of the constant map from $I$ to $c(i_0)$.
Set $m$ to be the maximal integer such that
$\pi_{n,m}(c)=\pi_{n,m}(\hat{c})$.
Choose $ r\in \ent/p^{\hat{e}}\ent$ such that $\delta(r,c(i_0))=m$.
Consider the element $c'\in \underset{e\in \cE}{\oplus}C(I_e,\ent/p^{e}\ent)$ defined as follows:
\begin{equation}
c'(i)=\left\{
\begin{array}{ll}
\pi_{\hat{e},e_i(n)}(r), & \hbox{if $e_i(n)>m$ and $m=\delta(c(i),c(i_0))$};\\
\pi_{\hat{e},e_i(n)}(c(i_0)), & \hbox{otherwise.}
\end{array}
\right .
\end{equation}
Then $c'\in C(L)\setminus D$.
\item Suppose that $k$ is a global field and $c\in C_{arith}(L)\setminus D$.
Then the $c'$ defined above is in $C_{arith}(L)\setminus D$.
\end{enumerate}
\end{lemma}
\begin{proof}
As $c$ is not in $D$, by the choice of $m$ there is some $i\in I$ such that $e_i(n)>m$ and
$\delta(c(i),c(i_0))=m$.
Hence $c'(i)\neq c'(i_0) \ ({\rm mod}\ p^{e_i(n)}\ent)$ by our construction and $c'\notin D$.

Now we show that $\pi_{n,l}(c')\in C^l(L)$ for $0\leq l\leq n$.
If $l\leq m$, then by the choice of $r$ we have $\pi_{n,l}(c(i_0))=\pi_{n,l}(r)$.
Clearly  $\pi_{n,l}(c')\in C^l(L)$.

Suppose $l> m.$
If $\Gamma_i^l\cap \Gamma_j^l\neq \varnothing,$  then by Remark \ref{e-e(n-m)}
$e_i(n)$ and $e_j(n)$ are at least $l$ and $c(i)=c(j)\ ({\rm mod }\ p^{l}\ent)$.
Hence $\delta(c(i_0),c(i))=m$ if and only if  $\delta(c(i_0),c(j))=m$.
By construction $c'(i)=c'(j)\ ({\rm mod }\ p^{l}\ent)$ and $\pi_{n,l}(c')\in C^l(L)$.

The proof of statement (2) is similar.

\end{proof}

\begin{lemma}\label{order and gamma}
Assume that $char(k)\neq p$.
Let $R$ be a discrete valuation ring  as in Notation \ref{galois groups and R}.
Denote by $\partial_R$ the residue map from $\Br(k(X_a))$ to $H^1(\kappa_R,\rat/\ent)$.
Suppose that  the order of $\partial_R(N_i,\tilde\chi)$ and
the order of $\partial_R(N_j,\tilde\chi)$ are both at least $p^{m}$.
Then $\Gamma_i^m\cap\Gamma_j^m$ is nonempty.
\end{lemma}

 \begin{proof}
 Let $\nu_R$ be the discrete valuation associated to $R$.
 Denote the completion of $k(X_a)$ with respect to $\nu_R$ by $k(X_a)_R$.
 Choose an extension $\omega_R$  of $\nu_R$ to a separable closure of $k(X_a)_R$.
 By the construction of $\tilde\chi$, $\partial_R(N_i,\tilde\chi)=\nu_R(N_i)\chi_R$. (See \cite{GS} 6.8.4 and 6.8.5.)
 Write $\nu_R(N_i)$ as $p^{m_i}q_i$  where $p\nmid q_i$.
 Let $p^{n_R}$ be the order of $\chi_R$.
 As the order of $\partial_R(N_i,\tilde\chi)\geq p^m$, we have $n_R-m_i\geq m$.

Since $\nu_R(N_i)=p^{m_i}q_i$, there is some factor  $M_i$ of $K_i\otimes_k k(X_a)_R$ such that  $p^{m_i+1}$ does not divide $\nu_R(N_{M_i/k(X_a)_R}(y_{M_i}))$, where $y_{M_i}$ is the projection of $y_i$ in $M_i$.
Let $\omega_{i,R}$ be the restriction of $\omega_R$ to $M_i$.
Write the  inert degree of $w_{i,R}$ over $\nu_R$ as  $p^{f_i}q'_i$ where $p\nmid q'_i$.
As $p^{m_i+1}$ does not divide $\nu_R(N_{M_i/k(X_a)_R}(y_{M_i}))$, we have
$f_i\leq m_i$.

Choose a factor $M$ of $K\otimes_k k(X_a)_R$ and let $\ol{M}$ be its residue field.
Let $\ol{M}_i$ be the residue field of $w_{i,R}$.
Both fields are considered as subfields of a separable closure $\kappa_R^s$ of $\kappa_R$.

As $f_i\leq m_i$ and $n_R-m_i\geq m$, the cyclic extension $\ol{M} \ol{M}_i/\ol{M_i}$ is of degree at least $p^{m}$.
Choose $g_R\in\cG_R$ such that $\chi_R(g_R)$ generates the image of $\chi_R$ in $\rat/\ent$.
Let $\cH_i$ be the subgroup of $\cG_R$ which fixes $\ol{M}_i$.
We claim that there are some $h_i\in \cG_R$  and
some $\sigma_R\in\langle h_i^{-1}g_R h_i\rangle\cap \cH_i$
such that $\chi_R(\sigma_R)$ is of order at least $p^m$.

Consider the group action of $\langle g_R \rangle$ on the set of left cosets of $\cH_i$ in $\cG_R$.
As $\lvert\cG_R/\cH_i\lvert= p^{f_i}q'_i$ with $p\nmid q_i'$ ,
there is some $h_i\in\cG_R$ such that $p^{f_i+1}$ does not divide the order of the orbit of  $ h_i\cH_i$.
Hence the stabilizer of $h_i\cH_i$ is $\langle g_R^{p^{f_i'}r_i}\rangle$ for some $f_i'\leq f_i$ and some
$ r_i$ coprime to $p$.
Let $\sigma_R=h_i^{-1}g_R^{p^{f_i'}r_i}h_i$.
Then $\chi_R(\sigma_R)=\chi_R(g_R^{p^{f_i'}r_i})$, which is of order $p^{n_R-f'_i}$.
Since $f'_i\leq f_i\leq m_i$ and $n_R-m_i\geq m$, the order of $\chi_R(\sigma_R)$ is at least $p^m$.

Let  $g$ and $\sigma$ be the image of $g_R$ and $\sigma_R$ in $G$.
Then $\sigma$ fixes $K_i$ and $\sigma$ is an element of order at least $p^m$ in $\Gal(K/k)$.
Hence the conjuacy class of  $g$ belongs to $\Gamma_i^m$.
The same argument proves that the conjuacy class of  $g$ belongs to $\Gamma_j^m$.
Hence $\Gamma^m_i\cap\Gamma^m_j$ is nonempty.

\end{proof}

\medskip

\medskip

Next we  prove that  for all $c \in C(L)$, the element  $\underset {i \in I} \sum c(i)(N_i,\tilde\chi)$ is unramified.

\begin{prop}\label{image is unramified}
Suppose that $char(k)\neq p$. The image of $u$ is an unramified subgroup of $\Br(k(X_a))$.
\end{prop}
\begin{proof}
By Lemma \ref{diagonal vanish} we can assume that $c(i)=0$ for some $i\in I_{\hat{e}}$.

Let $R$ be a discrete valuation ring of $k$ with residue field $\kappa_R$ containing $k$
and with field of fractions $k(X_a)$.
Let $\nu_R$ be the discrete valuation associated to $R$.
Denote by $\partial_R$ the residue map from $\Br(k(X_a))$ to $H^1(\kappa,\rat/\ent)$.
We claim that $u(c)=\underset{i\in I}{\sum} c(i)(N_i,\tilde\chi)$ is unramified at $R$.

Let $J(c)=\{i\in I \mid c(i)\neq 0\ {\rm in}\ \ent/p^{e_i(n)}\ent)\}$.
Let $m$ be the maximum integer such that $\pi_{n,m}(c)=0$ and
set $J_m(c)=\{i\in J(c) \mid m=\delta(0,c(i))\}$.
We prove by induction on $\lvert J(c)\rvert$.
If $|J(c)|=0$, then $c=0$ and our claim is trivial.
Suppose that our claim is true for $\lvert J(c)\rvert\leq h$.

Let $|J(c)|=h+1$. Then $c\notin D$ and $J_m(c)$ is nonempty.
Pick $j\in J_m(c)$ and choose $r\in \ent/p^{\hat{e}}\ent$
such that $c(j)=r$ $({\rm mod}\ p^{e_j(n)}\ent)$.
Let $c'$ be defined as in Lemma \ref{definition of c'}.
We first prove the $u(c')$ is unramified at $R$,
i.e. $\partial_R(u(c'))=0.$

Since $c(i)=0$, by the definition of $c'$ we have
$u(c')=\underset{s\in J_m(c)}{\sum}r(N_s,\tilde{\chi})$.
Hence  $\partial_R(u(c'))=\underset{s\in J_m(c)}{\sum}r\cdot\nu_R(N_s)\chi_R$.
Suppose that $\partial_R(u(c'))$ is not zero.
Then there is some $s\in J_m(c)$ such that $r\cdot\nu_R(N_s)\chi_R\neq 0$.
As $\delta(0,r)=m$,  the order of  $\nu_R(N_s)\chi_R$ is at least $p^{m+1}$.

By Lemma \ref{diagonal vanish}, there is some $t\in I\setminus J_m(c)$ such that
$r\cdot\nu_R(N_t)\chi_R(g)\neq 0$ and  the order of  $\nu_R(N_t)\chi_R$ is at least $p^{m+1}$.
By Lemma \ref{order and gamma}  the set $\Gamma_{s}^{m+1}\cap \Gamma_{t}^{m+1}$ is nonempty.
As $\delta(0,r)=m$, we have $c'(s)\neq c'(t)$ $({\rm mod} \ p^{m+1}\ent)$.
This contradicts that $c'\in C(L)$.
Therefore $\partial_R(u(c'))=0$ and $u(c')$ is unramified.

Consider the element $c-c'\in C(L)$.
By our construction of $c'$, the cardinality of $J(c-c')$ decreases by at least one.
By induction hypothesis $u(c-c')$ is unramified. Hence $u(c)$ is unramified.
\end{proof}

\bigskip

\noindent
{\bf Proof of Theorem \ref{main theo}.}
By Lemma \ref{diagonal vanish} and Proposition \ref{image is unramified}, the map
$$u:C(L)/D\to \Br(X^c_a)/\img(\Br(k))$$ is well-defined.

A similar argument as in \cite{BP}  Thm. 12.2 proves the injectivity of $u$.
Consider the projection from $X_a$ to the affine space $\aff^d$,
where $d=\underset{i\in I}{\sum}[K_i;k]$ and the coordinates are given by $x_{ij}$
defined in Notation \ref{definition of norm and Brauer group}.

Let $M$ be the function field of $\aff^d$.
Denote by $\chi_M$ the
image of $\chi$ in $H^1(M,\rat/\ent)$.
Suppose that $u(c)=\alpha\in\img(\Br(k))$.
Then $u(c)-\alpha$ is in the kernel of $\Br(M)\to\Br(k(X_a))$,which is generated by
$(a\underset{i\in I}{\prod} N_i^{-1}, \chi)$ (see \cite{BP} Lemma 12.3).
Therefore $u(c)-\alpha=r(a\underset{i\in I}{\prod} N_i^{-1}, \chi).$

Consider the discrete valuation $v_{N_i}$ on $M$ and let $\kappa_{N_i}$ be its residue field.
Denote by $\chi_{N_i}$ the image of $\chi$ in $H^1(\kappa_{N_i},\rat/\ent)$.
We claim that $\chi_{N_i}$ is of order $p^{e_i(n)}$.
Let $M_i$ be the function field of the subvariety of $\aff^{[K_i:k]}$ defined by $N_i$.
Then $\kappa_{N_i}=M_i(x_{jl})$ where $x_{jl}$ are defined as in Notation \ref{definition of norm and Brauer group} with $j\neq i$.
Hence $\kappa_{N_i}$ is purely transcendental over $M_i$.
Let $\chi_{M_i}$ be the image of $\chi$ in $H^1(M_i,\rat/\ent)$.
It suffices to  prove the order of $\chi_{M_i}$ is $p^{e_i(n)}$.
Note that $K\otimes_k M_i\otimes_k K_i\simeq K\otimes_k (\prod K_i(x))\simeq \prod E_i(n)(x)$,
where $x$ is a variable (see Notation \ref{prime power notation} for $E_i(n)$).
On the other hand $K\otimes_k M_i\otimes_k K_i\simeq (\prod \tilde{M}_i)\otimes_k K_i$,
where $\tilde{M}_i$ is a factor of $K\otimes_k M_i$.
If $[\tilde{M}_i:M_i]<p^{e_i(n)}$, then $\tilde{M}_i\otimes_k K_i=\tilde{M}_i\otimes_{M_i} M_i\otimes_k K_i$ is a product of extensions of $K_i(x)$ of degree less then $p^{e_i(n)}$, which is a contradiction.
Hence $[\tilde{M}_i:M_i]=p^{e_i(n)}$ and $\chi_{M_i}$ is of order $p^{e_i(n)}$.

As $\chi_{N_i}$ is of order $p^{e_i(n)}$, after taking residue of $u(c)-\alpha$ at $v_{N_i}$, we see that $c(i)=-r$ (mod $p^{e_i(n)}\ent$). Hence $c\in D$ and $u$ is injective.

Since $u$ is injective, $\lvert C(L)/D\rvert\leq\lvert \Br(X_a^c)/\img(\Br(k))\rvert$.
By Theorem \ref{general} and Theorem \ref{Colliot Br cycl}, the order of $C(L)/D$ is equal to the order of $\Br (T^c_{L/k})/\Br(k)$.
By Theorem \ref{x = t}, the map $u$ is surjective and hence is an isomorphism.

\medskip

We now prove the results announced in the introduction :

\begin{prop}\label{example of bicyclic}
Let $k$ be a field of $char(k)\neq p$.
Let $F$ be a bicyclic extension of $k$ with Galois group $\ent/p^n\ent\times\ent/p^n\ent$.
Let $K$ and $K_i$ be {linearly disjoint} cyclic subfields of $F$ with degree $p^n$
for $i=1,..., m$.
Then $$\Br(X_a^c)/\img(\Br(k))\simeq(\ent/p^{n}\ent)^{m-1},$$ and is generated by $(N_i,\chi)$ for $i=1...m-1$.
\end{prop}

\begin{proof}
There is a number field $\ell$ and an unramified Galois extension $\ell'/\ell$ such that
$\Gal(\ell'/\ell)\simeq \Gal(F/k)$ (\cite{F}).
Set $F = (\ell')^{G_K}$, $L_i = (\ell')^{G_{K_i}}$ and
$E = F \times \underset{i \in I} \prod L_i$. By construction, the $G$-lattices $\hat T_{E/\ell}$ and $\hat T_{L/k}$ are isomorphic.

By \cite{Lee} Proposition 7.3, the group $\sha^2_\omega(\ell,\hat{T}_{E/\ell})\simeq (\ent/p^n\ent)^{m-1}$.
Since $\sha^2_\omega(\ell,\hat{T}_{E/\ell})\simeq\sha^2_{\rm cycl}(G,\hat{T}_{E/\ell})\simeq
\sha^2_{\rm cycl}(G,\hat{T}_{L/k})$, we have $C(L)/D\simeq(\ent/p^n\ent)^{m-1}$ by Theorem \ref{general}.
The assertion then follows from Theorem \ref{main theo}.

\end{proof}

\section{Global fields}\label{global section}

We keep the notation of the previous section, and in addition we assume that $k$ is a global field.
Denote by $\ruB(X_a^c)$  the subgroup of $\Br(X_a^c)/\img(\Br(k))$ consisting of locally trivial elements; and by $\ruB_\omega(X_a^c)$ the subgroup consisting of elements which are trivial at almost all places of $k$.

\begin{theo}\label{main theo sha}
Suppose that $k$ is a global field with $char(k)\neq p$.
Then
\begin{enumerate}
\item $u$ induces an isomorphism between $C(L)/D$ and $\ruB_\omega(X_a^c)$.
\item $u$ induces an isomorphism between $C_{arith}(L)/D$ and $\ruB(X_a^c)$.
\end{enumerate}
\end{theo}

\begin{proof}
First note that \cite{Sansuc} 6.1.4 remains true over global fields.
Hence $\ruB_\omega(X_a^c)\simeq \ruB_\omega(X_a)$ and $\ruB(X_a^c)\simeq \ruB(X_a)$.

By \cite{Sansuc}  6.8 and 6.9 (ii), we have $\ruB(X_a)\simeq\sha^2(k,\hat{T}_{L/k})$
(resp.$\ruB_\omega(X_a)\simeq\sha_{\omega}^2(k,\hat{T}_{L/k})$).
By Theorem \ref{general global} and Corollary \ref{C and C omega}, we conclude that
$C(L)/D\simeq\ruB_\omega(X_a^c)$ and $C_{arith}(L)/D\simeq \ruB(X_a^c)$.

As $\ruB_\omega(X_a^c)$ is a subgroup of $\Br(X_a^c)/\img(\Br(k))$ with the same cardinality,
the first statement follows from Theorem \ref{main theo}.

To see that $u$ gives rise to the desired isomorphism in (2), it is sufficient to show
that $u(c)_v\in \Br(X_a^c\times_k k_v)$ is in the image of $\Br(k_v)$ for $c\in C_{arith}(L)$
 and for $v\in V_k$.

By Lemma \ref{diagonal vanish} we can assume that $c(i_0)=0$ for some $i_0\in I_{\hat{e}}$.
Let $J(c)=\{i\in I \mid c(i)\neq 0\ {\rm in}\ \ent/p^{e_i(n)}\ent)\}$.
Set $m$ be the maximum integer such that $\pi_{n,m}(c)=0$ and
set $J_m(c)=\{i\in J(c) \mid m=\delta(0,c(i))\}$.
We prove by induction on $\lvert J(c)\rvert$.
If $|J(c)|=0$, then $c=0$ and our claim is trivial.
Suppose that our claim is true for $\lvert J(c)\rvert\leq h$.

Let $|J(c)|=h+1$. Then $c\notin D$ and $J_m(c)$ is nonempty.
Pick $j\in J_m(c)$ and choose $r\in \ent/p^{\hat{e}}\ent$
such that $c(j)=r$ $({\rm mod}\ p^{e_j(n)}\ent)$.
Let $c'$ be defined as in Lemma \ref{definition of c'}.

For $v\in V_k$, let $\chi_v$ be the image of $\chi$ in $H^1(k_v,\rat/\ent)$, and $\chi_{i,w}$ be the image of $\chi$
in $H^1((K_{i})_w,\rat/\ent)$ where $w$ is a place of $K_i$.

By the definition of $C_{arith}(L)$, the set $V^{m+1}_{i}\cap V^{m+1}_j$ is empty
for any $i\in J_m(c)$ and for any $j\notin J_{m}(c)$.

Let $v\in V_k$.
Suppose that $v\notin \underset{i\in J_m(c)}{\cup} V^{m+1}_i$.
Then for all $i\in J_m(c)$, $\chi_{i,w}$ is of order at most $m$
for all $w\mid v$.
By the projection formula,  $(N_i,\chi)_v$ has order at most $p^{m}$.
Hence $c'(i)(N_i,\chi)_v=0$ and  $u(c')_v=0$ in this case.

Suppose that $v\in V^{m+1}_i$ for some $i\in J_m(c)$.
Let $d\in D$ be the image of the constant map from $I$ to $r$.
Set $\ol{c}=d-c'$.
As the set $V^{m+1}_{i}\cap V^{m+1}_j$ is empty for any $j\notin J_{m}(c)$,
$v\notin \underset{j\notin J_m(c)}{\cup} V^{m+1}_j$.
The same argument shows that $u(\ol{c})_v=0$.
Hence $u(c')_v$ is  in the image of $\Br(k_v)$.

Since the cardinality of $J(c-c')$ decreases by at least one, by induction hypothesis $u(c-c')\in\ruB(X_a^c)$.
In combination with $u(c')\in\ruB(X_a^c)$, we see that $u(c)$ is in $\ruB(X_a^c)$.

\end{proof}

\begin{example}\label{2-power-bicyclic}
Let $k=\rat(i)$.
Let $K=k(\sqrt[4]{17})$, $K_1=k(\sqrt[4]{17\times 13})$ and $K_2=k(\sqrt[4]{13})$.
By \cite{Lee}  Example 7.4, $\sha^2_\omega(k,\hat{T}_{L/k})\simeq\ent/4\ent$ and
$\sha^2(k,\hat{T}_{L/k})\simeq\ent/2\ent$.
By Theorem \ref{main theo} and Theorem \ref{main theo sha}, the element $(N_2,\chi)$ generates the group $\Br(X_a^c)/\img(\Br(k))$ and $2(N_2,\chi)$ generates $\ruB(X_a^c)$.
\end{example}

More generally we have the following.
\begin{prop}\label{example of bicyclic global}
Let $k$ be a global field of $char(k)\neq p$.
Let $F$ be a bicyclic extension of $k$ with Galois group $\ent/p^n\ent\times\ent/p^n\ent$.
Let $K$ and $K_i$ be {linearly disjoint} cyclic subfields of $F$ with degree $p^n$
for $i=1,..., m$.
Assume moreover that $F\otimes_k k_v$ is a product of cyclic extensions for all $v\in V_k$. Then
$(N_i,\chi)$ generates $\ruB(X_a^c)$.
\end{prop}

\begin{proof}
By  \cite{Lee}  Prop. 7.3, we have $\sha^2(k,\hat{T}_{L/k})\simeq (\ent/p^n\ent)^{m-1}$.
We apply Theorem \ref{general global} (1) and Theorem \ref{main theo sha} (2) to conclude.
\end{proof}

\section{An application to Hasse principles}\label{Hasse principles}
In this section we apply Theorem \ref{main theo} to give an alternative proof of
\cite{BLP} Theorem 7.1 for
 $k$  a global field with $char(k)\neq p$.
 Moreover we can assume that $K/k$ is a cyclic extension of $p$-power degree. (See \S2 and \cite{BLP} 6.3.)

We use the notation of the previous sections. In particular, $X_a$ is the affine variety defined in the introduction,
$K/k$ is a cyclic extension of $p$-power degree, and $K_i/k$ is a finite separable extension for all $i \in I$.  Recall that
 $\chi$  an injective morphism from ${\rm Gal}(K/k)$ to $\rat/\ent$.

Let $\chi_v$ be the image of $\chi$ in $H^1(k_v,\rat/\ent)$.
Let $\inv$ be the Hasse invariant map ${\rm inv}:\Br(k_v)\to\rat/\ent$.

Denote by $K_i^v$ the algebra $K_i\otimes_k k_v$.
Suppose that there is a local point $(x_i^v)\in\underset{v\in V_k}{\prod} X_a(k_v)$, where $x_i^v\in K_i^v$ for $i\in I$ and $x_0^v\in K\otimes_k k_v$.
Define $\alpha_a:C_{arith}(L)/D\to\rat/\ent$ as
\[\alpha_a(c)=\underset{v\in V_k}{\sum} \underset{i\in I}{\sum}c(i){\rm inv}(N_{K_i^v/k_v}(x_i^v),\chi_v)\]
\begin{theo}\label{Hasse principle}
Suppose that there is a local point $(x_i^v)\in\underset{v\in V_k}{\prod} X_a(k_v)$.
Then the map $\alpha_a$ is the Brauer-Manin pairing and $X_a$ has a $k$-point if and only if $\alpha_a=0$.
\end{theo}

\begin{proof}
First we consider the case where $k$ is a number field.
By Theorem \ref{main theo sha} and \cite{Sansuc} Lemma 6.2, the map $\alpha_a$ is the Brauer-Manin pairing of $X_a^c$.
Our claim then follows from Sansuc's result \cite{Sansuc} Cor. 8.7.

For $k$ a global function field, we apply Theorem \ref{main theo sha} and
\cite{DH}  Theorem 2.5 to conclude.
\end{proof}

\bigskip
\bigskip
Eva Bayer--Fluckiger

EPFL-FSB-MATH

Station 8

1015 Lausanne, Switzerland

\medskip

eva.bayer@epfl.ch

\bigskip

Ting-Yu Lee

NTU-MATH

Astronomy Mathematics Building 5F,

No. 1, Sec. 4, Roosevelt Rd.,

Taipei 10617, Taiwan (R.O.C.)

\medskip
tingyulee@ntu.edu.tw


\bigskip


\begin{thebibliography}{99}


\bibitem[BLP 19]{BLP} E. Bayer-Fluckiger, T-Y. Lee, R. Parimala,  \textit{ Hasse principles for multinorm equations}, Adv.
Math.
\textbf{356} (2019).

\bibitem[BP 20]{BP} E. Bayer-Fluckiger, R. Parimala, \textit{
On unramified Brauer groups of torsors over tori}, Doc. Math. {\bf 25} (2020),  1263-1284.


\bibitem[C 19]{C} K. Cesnavicius, \textit{Purity for the Brauer group}, Duke Math. J. {\bf 168} (2019), 1461-1486.

\bibitem[CTHSk 03]{CTHSk 03} J-L. Colliot-Th\'el\`ene, D. Harari,  A. Skorobogatov,
\textit{Valeurs d’un polyn\^{o}me \`a une variable repr\'esent\'ees par une norme},
Number theory and algebraic geometry, London Math. Soc. Lecture Note Ser. {\bf 303}, Cambridge Univ. Press, Cambridge (2003), 69-89.

\bibitem[CTHSk 05]{CTHSk 05} J-L. Colliot-Th\'el\`ene, D. Harari,  A. Skorobogatov, \textit{Compactification \'equivariante d’un tore} (d’apr\`es Brylinski
et K\"unnemann), Expo. Math {\bf 23} (2005), 161-170.

\bibitem[CTSk 19]{CTSk} J-L. Colliot-Th\'el\`ene, A. Skorobogatov, \textit {The Brauer-Grothendieck group}, Springer International Publishing (2021).

\bibitem[CTS 77]{CTS 77} J-L. Colliot-Th\'el\`ene, J-J. Sansuc, \textit{La R-\'equivalence sur les tores}, Ann. Sc ENS \textbf {10} (1977),
175-229.
\bibitem[CTS 87]{CTS 87} J-L. Colliot-Th\'el\`ene, J-J. Sansuc, \textit{Principal homogeneous spaces under flasque tori~: applications}, J. Algebra \textbf {106} (1987),
148-205.

\bibitem[DH 17]{DH} C. Demarche and D. Harari, \textit{Local-global principles for homogeneous spaces of reductive groups over global fields}, preprint, https://arxiv.org/abs/2107.08906.

\bibitem[F 62]{F} A. Fr\"ohlich, \textit{On non-ramified extensions with prescribed Galois group}, Mathematika \textbf{9} (1962), 133-134.



\bibitem[GS 06]{GS} P. Gille, T. Szamuely,
\textit{Central simple algebras and Galois cohomology}, Cambridge University Press (2006).

\bibitem[Lee 21]{Lee} T-Y. Lee, \textit{The Tate-Shafarevich groups of multinorm-one tori},
to appear in J. Pure  Appl. Algebra.



\bibitem[Sa 81]{Sansuc} J.-J. Sansuc, \textit{Groupe de Brauer et arithm\'etique des groupes alg\'ebriques lin\'eaires sur un corps de nombres},
J. Reine Angew. Math. \textbf {327} (1981) 12–80.

\bibitem[Se 79] {corps locaux} J-P. Serre, \textit{Local fields},
Graduate Texts in Mathematics {\bf 67} Springer-Verlag, New York-Berlin (1979).

\end{thebibliography}
\end{document}